\documentclass{article}

\usepackage{amsopn}
\usepackage{amsfonts} 
\usepackage{amsmath, amsthm}
\usepackage{amssymb} 
\usepackage{amsthm} 
\usepackage{amscd}
\usepackage{enumerate}
\usepackage{url}
\usepackage{xspace}

\newcommand{\Babs}[1]{\Bigl\lvert #1\Bigr\rvert}
\newcommand{\nabs}[1]{\lvert #1\rvert}
\newcommand{\babs}[1]{\bigl\lvert #1\bigr\rvert}

\newcommand{\nsset}[1]{\{#1\}} 
\newcommand{\bsset}[1]{\bigl\{#1\bigr\}}

\newcommand{\nparen}[1]{(#1)}
\newcommand{\bparen}[1]{\bigl(#1\bigr)}
\newcommand{\Bparen}[1]{\Bigl(#1\Bigr)}
\newcommand{\map}[3]{#1 \colon #2 \rightarrow #3}
\newcommand{\binlang}{\nsset{0,1}^*}
\newcommand{\nbrack}[1]{[#1]}
\newcommand{\bbrack}[1]{\bigl[#1\bigr]}
\newcommand{\nfloor}[1]{\lfloor #1\rfloor}

\newcommand{\qtext}[1]{\quad\text{#1}\quad}
\newcommand{\qqtext}[1]{\qquad\text{#1}\qquad}
\newcommand{\ntoinf}{n \rightarrow \infty}
\newcommand{\ncset}[2]{\{\,#1\,\colon\,#2\,\}} 
\newcommand{\bcset}[2]{\bigl\{\,#1\,\colon\,#2 \,\bigr\}}

\newcommand{\calA}{\ensuremath{\mathcal{A}}\xspace}

\newcommand{\calS}{\ensuremath{\mathcal{S}}\xspace}
\newcommand{\calT}{\ensuremath{\mathcal{T}}\xspace}

\newcommand{\nnorm}[1]{\lVert #1\rVert}
\newcommand{\binm}{\nsset{0,1}^*}
\newcommand{\bfu}{\ensuremath{\mathbf{u}}\xspace}
\newcommand{\bfv}{\ensuremath{\mathbf{v}}\xspace}
\newcommand{\bfw}{\ensuremath{\mathbf{w}}\xspace}

\newcommand{\bfx}{\ensuremath{\mathbf{x}}\xspace}
\newcommand{\bfy}{\ensuremath{\mathbf{y}}\xspace}
\newcommand{\bfz}{\ensuremath{\mathbf{z}}\xspace}
\newcommand{\N}{\ensuremath{\mathbb{N}}\xspace}

\newenvironment{enumeratei}{\begin{enumerate}[\upshape (i)]}{\end{enumerate}}
\theoremstyle{definition}
\newtheorem{definition}{Definition}
\newtheorem{theorem}[definition]{Theorem}
\newtheorem{corollary}[definition]{Corollary}
\newtheorem{remark}[definition]{Remark}
\newtheorem{lemma}[definition]{Lemma}
\newtheorem{openproblem}{Open problem}

\newcommand{\thue}{\mathbf{TM}_0}

\begin{document}

\title{Avoiding Abelian powers in binary words with bounded Abelian complexity}

\author{Julien Cassaigne
\footnote{Institut de Math\'ematiques de Luminy,
case 907, 163 avenue de Luminy, 13288 Marseille Cedex 9, France  ({\tt cassaigne@iml.univ-mrs.fr})}
\and
Gw\'ena\"el Richomme
\footnote{Univ. Paul-Val\'{e}ry Montpellier 3, UFR IV, Dpt MIAp, Case J11, Rte de Mende, 34199 Montpellier Cedex 5, France.
LIRMM (CNRS, Univ. Montpellier 2), UMR 5506 - CC 477, 161 rue Ada,  34095 Montpellier Cedex 5, France ({\tt gwenael.richomme@univ-montp3.fr})}
\and
Kalle Saari
\footnote{\textbf{Corresponding author}. Department of Mathematics, University of Turku, FI-20014, Turku, Finland ({\tt kasaar@utu.fi})}
\and
Luca Q. Zamboni
\footnote{Universit\'e de Lyon, Universit\'e Lyon 1, CNRS UMR 5208 Institut Camille Jordan, B\^atiment du Doyen Jean Braconnier, 43, blvd du 11 novembre 1918, F-69622 Villeurbanne Cedex, France
({\tt zamboni@math.univ-lyon1.fr}).
Reykjavik University, School of Computer Science, Kringlan 1, 103 Reykjavik, Iceland ({\tt lqz@ru.is}).}
}

\maketitle

\begin{abstract}
The notion of Abelian complexity of infinite words was recently used by the three last authors to investigate various Abelian 
properties of words. In particular, using van der Waerden's theorem, they proved that if a word avoids Abelian $k$-powers for some integer $k$, then  its Abelian complexity is unbounded. This suggests the following question: How frequently do Abelian $k$-powers occur in a word having bounded Abelian complexity? In particular, does every uniformly recurrent word  having bounded Abelian complexity begin in an Abelian $k$-power? While this is true for various classes of uniformly recurrent words, including for example the class of all Sturmian words, in this paper we show the existence of uniformly recurrent  binary words, having bounded Abelian complexity, which admit an infinite number of suffixes which do not begin in an Abelian square.  We also show that the shift orbit closure of any infinite binary overlap-free word contains a word which avoids Abelian cubes in the beginning.
We also consider the effect of morphisms on Abelian complexity and show that the morphic image of a word having bounded Abelian complexity has bounded Abelian complexity. 
Finally, we give an open problem on avoidability of Abelian squares in infinite binary words and show that it is equivalent to a well-known open problem of Pirillo--Varricchio and
Halbeisen--Hungerb\"{u}hler.
\end{abstract}

\indent {\textbf{keywords:} Avoidability in words, Abelian power, Abelian complexity.}

\section{Introduction}

The notion of Abelian complexity of an infinite word was recently developed by the three last authors~\cite{RicSaaZam2010} to study various Abelian properties of words. This gave way to several other interesting results:
While aperiodic balanced binary (i.e., Sturmian) and ternary words are examples of recurrent words with constant Abelian complexity~2 and~3, respectively (see ~\cite{RicSaaZam2010}),  Currie and Rampersad~\cite{CurRam2009} showed that there are no recurrent words with constant Abelian complexity $k\geq4$. On the other hand,
Saarela~\cite{Saarela2010} showed that for any $k\geq 2$ there exist recurrent infinite words whose Abelian complexity is ultimately constant~$k$. 

It is natural to try to determine the Abelian complexity of some well-known infinite words. For Sturmian words, this was essentially done already by Morse and 
Hedlund~\cite{MorHed1940} by showing that these words are balanced. Characterizing the Abelian complexity of the Tribonacci word is substantially harder, and this was only partially achieved by the 
three last authors~\cite{RicSaaZam2010b}. At any rate, the Abelian complexity  of the Tribonacci word is bounded by~7. A similar result for a different class of ternary words was obtained very recently by 
Turek~\cite{Turek2010}. 

Before proceeding any further, let us quickly recall some basic terminology.
Let $A$ be a finite alphabet. 
For each word $x\in A^*$, we denote the length of $x$  by $\nabs{x}$ and  the number of  occurrences of a letter $b\in A$ in $x$ by $\nabs{x}_b$.
If $A = \nsset{a_1, a_2, \ldots, a_k}$,  the \emph{Parikh vector} of  $x$ is $ \Psi(x) := \nparen{\nabs{x}_{a_1}, \ldots, \nabs{x}_{a_k}}$.
If two words $x,y \in A^*$ have the same Parikh vector, then we say that they are \emph{Abelian equivalent}, and we express this relation by writing $x \sim_{ab} y$.
The empty word is denoted by~$\epsilon$.

An \emph{infinite word} is an expression  $\bfx = a_0 a_1 a_2 a_3\cdots$, where $a_i \in A$; occasionally we index infinite words starting from 1 instead of~0.
An infinite word of the form $a_{n} a_{n+1} a_{n+2}\cdots$, where $n\geq 1$, is called a \emph{suffix} of~$\bfx$; 

A finite word $x$ is a \emph{subword} of $\bfx$ if we have  $x = a_n a_{n+1} \cdots a_{n+\nabs{x} - 1}$ for some integer $n\geq 1$.
An infinite word $\bfx$ is called \emph{$C$-balanced}, for some $C>0$,  if all subwords $x,y$ of $\bfx$ with $\nabs{x} = \nabs{y}$ satisfy 
$\babs{\nabs{x}_{b} - \nabs{y}_{b}} \leq C$ for all letters $b\in A$. If $C=1$, then we simply say that $\bfx$  is \emph{balanced}.

The \emph{subword complexity} of $\bfx$ is the mapping $\map{\rho_\bfx}{\N}{\N}$ for which $\rho_\bfx(n)$ is the number of distinct subwords of length~$n$.
The \emph{Abelian complexity} of $\bfx$ is the mapping $\map{\rho_\bfx^{ab}}{\N}{\N}$ for which $\rho_\bfx^{ab}(n)$ is the number of subwords of length~$n$ that are pairwise Abelian inequivalent.
In~\cite{RicSaaZam2010}, the three last authors showed that Abelian complexity is closely linked to the notion of balance:

\begin{lemma}[\cite{RicSaaZam2010}] \label{220220101341}
An infinite word has bounded Abelian complexity if and only if it is $C$-balanced for some $C>0$.
\end{lemma}

A nonempty word $w$ is an \emph{Abelian $k$-power}, where $k\geq 2$ is an integer, if it can be written in the form $w = u_1 u_2 \cdots u_k$ with the $u_i$ pairwise Abelian equivalent. Then we also say that $w$ has (Abelian) period $\nabs{u_{1}}$.
According to a classical result, an infinite word is ultimately periodic if and only if its subword complexity is bounded.
The next result, which relies on the well-known van der Waerden's theorem, may be considered as a partial Abelian analogue of this.

\begin{theorem}[\cite{RicSaaZam2010}]
If an infinite word has bounded Abelian complexity, then it contains Abelian $k$-powers for all integers $k\geq1$. 
\end{theorem}
\noindent  This result naturally gives rise to the following question:
\begin{quote}
\textit{Does every uniformly recurrent word with bounded Abelian complexity begin in an Abelian $k$-power for every $k$?}
\end{quote}
It turns out that the Thue--Morse word, or more generally the infinite binary overlap-free words, nicely shed light to this question.
Recall that the Thue--Morse word $\thue = t_0 t_1 t_2 \cdots $ is the fixed point of the morphism $\map{\mu}{\binlang}{\binlang}$ starting with the letter~0; the binary overlap-free words are characterized by the absence of subwords of the form $axaxa$, where $a\in \binlang$.
As is well-known, the letter $t_n$  is given by the number of bits 1 in the binary expansion of $n$  modulo~2.
The Abelian complexity of the Thue--Morse word is obtained as a corollary of the following result.

\begin{theorem}[\cite{RicSaaZam2010}]\label{220220101336}
An aperiodic infinite word $\bfx$ has Abelian complexity
\[
\rho_{\bfx}^{{ab}}(n) =
\begin{cases}
2 & \text{if $n$ is odd;} \\
3 & \text{if $n$ is even} \\
\end{cases}
\]
if and only if $\bfx$ is of the form $0\mu(\bfy)$, $1\mu(\bfy)$, or $\mu(\bfy)$, where $\mu$ is the Thue--Morse morphism and $\bfy$ is any infinite binary word.
\end{theorem}

In particular, the Thue--Morse word has a bounded Abelian complexity,  and as a fixed point of a primitive morphism, it is also uniformly recurrent.
In Section~\ref{201003121021} we show that while every suffix of the Thue--Morse word begins in an Abelian $k$-power, 
for $k\geq 3$, there is no upper bound for the length of the shortest such powers.  It follows that there exist uniformly recurrent overlap-free binary words with bounded Abelian complexity which do not begin in an Abelian cube (i.e., an Abelian $3$-power). So this already answers our original question, and incidentally, along the way it establishes a conjecture made by the second author in his PhD thesis~\cite{Saari2008} namely that the Thue--Morse word is so-called everywhere Abelian 2-repetitive, but not 3-repetitive.

Avoiding Abelian cubes in the beginning, however, is not the best possible answer.
In Section~\ref{201003121129}, as a prelude to the remaining  sections, we give an example of a uniformly recurrent binary word with bounded Abelian complexity that avoids Abelian squares (i.e., 2-powers) in the beginning. We then study two questions on the relation between morphisms and infinite words with bounded Abelian complexity.  
In Section~\ref{201003121131}, we characterize the morphisms which, like the Thue-Morse morphism, map any infinite word to a word with bounded Abelian complexity, 
and in Section~\ref{201003121130}, we show that although the image under a morphism of a word can have unbounded Abelian complexity, all morphisms preserve bounded Abelian complexity.

Then in Section~\ref{201003121132}  we construct an example of a uniformly recurrent binary  word with bounded Abelian complexity avoiding Abelian squares in infinitely many positions.
Lastly, in Section~\ref{201003121133}, we explain why this result is optimal and conclude with an open problem
showing that a slight variation of this is equivalent to an older one raised independently by Pirillo--Varricchio~\cite{PirVar1994} and 
Halbeisen--Hungerb\"{u}hler~\cite{HalHun2000}.

Let us finally remark that some other questions on the avoidance of Abelian powers in binary words have been considered before:
Entringer, Jackson, and Schatz~\cite{EntJacSch1974} showed that every infinite binary word contains arbitrarily long Abelian squares, and
Dekking~\cite{Dekking1979} constructed an infinite binary word that avoids Abelian 4-powers.
It is not known whether there exist an infinite binary word avoiding long Abelian cubes (see Section~2.9 in~\cite{Shallit2009}).

\section{Abelian cubes and overlap-free words}\label{201003121021}

In this section, we let  $\nbrack{x}_2$ denote the binary expansion of a positive integer~$x$ without leading zeros.
For example, $\nbrack{2^n}_2 = 10^n$ for all  integers~$n\geq0$. 

\begin{theorem}
Every suffix of the Thue--Morse word $\thue$ begins in an Abelian $k$-power for all positive integers~$k$.
\end{theorem}
\begin{proof}
Let us denote $\thue = t_0 t_1 t_2 \cdots$.
We first observe that there exists a positive integer $m$ such that  
\[
t_0 = t_{m} = t_{2m} = \cdots = t_{(k+1)m} \, (= 0).
\] 
For example, we may take $m=2^{\nfloor{\log_2(k+1)} +1} + 1$. 
This follows from the fact that the number of 1s in the expansion $[i \cdot m]_2$ is twice the number of 1s in the expansion $[i]_2$

Another observation we need is that if $x$ and $y$ are any binary words of the same length, then the words $\mu(x)$ and $\mu(y)$ are Abelian equivalent.

Now let $\bfx$ be a suffix of $\thue$, so that we have $\thue = p\, \bfx$ for some~$p$. Choose an integer $n\geq 1$ such that $\nabs{\mu^n(0)} > \nabs{p}$. Since $\mu^n(0)$ is a prefix of $\thue$,  there exists a word $s$ such that $\mu^n(0) = ps$. By the observation we made in the beginning of the proof, there exist words 
$x_1, x_2, \ldots, x_k$ of the same length such that the word
\[
\bparen{\prod_{i=1}^k 0 x_i }0 =   0x_1 0x_2 0\cdots 0 x_k 0
\]
is  a prefix of $\thue$. But  then so is the word
\[
\mu^n\bparen{0x_1 0x_2 0\cdots 0 x_k 0} = ps \mu^n(x_1) ps \mu^n(x_2) ps \cdots ps \mu^n(x_k) ps.
\]
Therefore the suffix $\bfx$ begins in an Abelian $k$-power
\[
s \mu^n(x_1) p. \, s \mu^n(x_2)p. \, \ldots  \, . s \mu^n(x_k) p. 
\]
\end{proof}

Now we prove the existence of a binary overlap-free word that does not begin in an Abelian cube (see Corollary~\ref{103125012010}).

\begin{lemma}\label{100120102316}
Let $\bfx$ and $\bfz=z_{0} z_{1} z_{2} \ldots$ be infinite words such that $\bfx = \mu(\bfz)$, and let $n\geq 0$ be an integer. If the position $2n+ 1$  in \bfx has an occurrence of an Abelian $3$-power with period $\ell$, then 
$\ell = 2k$ for some $k$  and we have
\[
z_{n} = z_{n+k} = z_{n + 2k}.
\]
\end{lemma}
\begin{proof}
Let $w$ denote an Abelian 3-power starting in position $2n + 1$ in $\bfx$, and let $\ell$ denote  its period.
First we show that~$\ell$ is an even integer. To this end, suppose that $\ell$ is odd.
Then since $w$ occurs in an odd index, it is of the form $w = a \mu(x). \mu(y)b. c \mu(z)$,
where $a,b,c \in \nsset{0,1}$ and the words $a\mu(x)$, $\mu(y)b $, and $c\mu(z)$ are Abelian equivalent.
Since each of the words $x,y$, and $z$ are of the same length,  their images $\mu(x)$, $\mu(y)$, and $\mu(z)$ are trivially Abelian equivalent.
Therefore we have $a=b=c$. But this is not possible because  $bc$ is an image of a letter so that $bc \in \nsset{01,10}$.

Thus we may suppose that $\ell = 2 k$ for some positive integer~$k$. Then we have
\[
w = a \mu(x)b. c\mu(y)d. e \mu(z)f
\]
for some letters $a,b,c,d,e$, and $f$ with $bc, de \in \nsset{01,10}$.
Furthermore, each of the words  $a \mu(x)b$, $c \mu(y)d$, and $e \mu(z)f $  is of length $2k$, and they are pairwise Abelian equivalent.
Now we have $a\neq b$;  for otherwise the fact that $b\neq c$ implies that the words $a\mu(x) b$ and $c \mu(y) d$ are not Abelian equivalent.
Similarly we see that $c\neq d$ and $e \neq f$, and thus we conclude that either
\[
w = 0 \mu(x)1. 0\mu(y)1. 0 \mu(z)1 \qqtext{or} w = 1\mu(x)0. 1\mu(y)0 .  1\mu(z)0
\]
In the first case, the occurrence of $w$ extends to the left by the letter 1, so that the word $1w$ occurs at position $2n$.
It follows that the word $1x1y1$ occurs at position~$n$. Since $\nabs{1x} = \nabs{1y} = k$, we therefore have $z_n = z_{n+k} = z_{n+2k} = 1$.
In the second case we deduce similarly that $z_n = z_{n+k} = z_{n+2k} = 0$.
\end{proof}

\begin{lemma} \label{100120102311}
Let $n$ and $k$ be integers with $n\geq1$ and $0 \leq k < 2^n$. Then there exists a word $u\in \nsset{0,1}^{n}$ such that
\[
\bbrack{2^{n+1} +  k }_2 = 10u \qqtext{and} \bbrack{2^{n+1} + 2 k + 1 }_2 = 1u1.
\]
\end{lemma}
\begin{proof}
Since $0  \leq k \leq 2^n - 1$, we have
\[
2^{n+1} \leq 2^{n+1} +  k  \leq 2^{n+1} + \nparen{2^n - 1};
\]
in other words
\[
10^{n+1} \leq_{\text{lex}} \bbrack{2^{n+1}  + k}_2 \leq_{\text{lex}} 101^{n},
\]
where $\leq_{\text{lex}}$ denotes the lexicographic order. Therefore $\bbrack{2^{n+1}  +  k}_2 = 10u$ for some $u\in \nsset{0,1}^n$.
Thus we have
\begin{align*}
\bbrack{2^{n+1} + k}_2 &= 10u\phantom{0} \qtext{implying} \\
\bbrack{2\bparen{2^{n+1}  + k}}_2 &= 10u0 \qtext{implying} \\
\bbrack{2\bparen{2^{n+1}  + k} - 2^{n+1} }_2 &= 1u0\phantom{0} \qtext{implying} \\
\bbrack{2\bparen{2^{n+1}  + k} - 2^{n+1} + 1 }_2 &= 1u1\phantom{0} \qtext{implying} \\
\bbrack{2^{n+1}  + 2k + 1}_2  &= 1u1.
\end{align*}
\end{proof}

\begin{lemma}\label{120120101044}
Let $n\geq 1$ be an integer. If an Abelian 3-power occurs in $\thue$ at the position $2^{n+2} - 1$, then its period is at least~$2^{n+1}$.
\end{lemma}
\begin{proof}
Suppose that $w$ is an Abelian 3-power with period $\ell<2^{n+1}$  occurring at position $2^{n+2} - 1$. Then by Lemma~\ref{100120102316}, we have that
$\ell = 2k$ for some positive integer $k$, and
\[
t_{2^{n+1} - 1} = t_{2^{n+1} - 1 + k } = t_{2^{n+1} - 1 + 2 k }.
\]
Thus in particular the number of~1s occurring in the expression $\bbrack{2^{n+1} - 1 + k }_2$ has the same parity as
the number of~1s occurring in  $\bbrack{2^{n+1} - 1 + 2 k}_2$. 
Denoting $k = j + 1$, the inequality $\ell < 2^{n+1}$ gives $0 \leq j < 2^{n}$, which contradicts Lemma~\ref{100120102311}.
\end{proof}


An infinite word $\bfz$ is called \emph{everywhere Abelian $k$-repetitive} if there exists an integer $n\geq 1$ such that every subword of length $n$ has a prefix that is an Abelian $k$-power. All Sturmian words, for example, are everywhere Abelian $k$-repetitive for all integers $k$~\cite{RicSaaZam2010}. 
The first item in the next result proves a conjecture by the third author~\cite{Saari2008}, and the second item shows that Abelian cubes need not occur in every position even in a uniformly recurrent binary word with bounded Abelian complexity.

\begin{corollary}\label{103125012010}
The following three statements hold:
\begin{enumeratei}
\item The Thue-Morse word $\thue$ is not Abelian 3-repetitive.\label{083525012010}
\item The words $0\thue$ and $1\thue$ do not have an Abelian cube as a prefix.\label{083625012010}
\item All binary overlap-free infinite words are Abelian 2-repetitive but not Abelian 3-repetitive. \label{083725012010}
\end{enumeratei}
\end{corollary}
\begin{proof}
Lemma~\ref{120120101044} says that the length of the shortest Abelian 3-power occurring  at positions of the form   $2^{n} - 1$ grows arbitrarily large. 
Therefore the Thue-Morse word is not Abelian 3-repetitive, and \eqref{083525012010} is proved.

The second item follows from Lemma~\ref{120120101044} since the word $a\mu^{n}(1)$ occurs at position $2^{n}-1$ where $a$ is 0 or 1 depending on the parity of~$n$.
Note that $\thue$ is obtained from $\lim_{n\rightarrow\infty}\mu(1)$ by exchanging 0s and 1s.

Let us then prove item~\eqref{083725012010}. Firstly, it is readily verified that  every  binary overlap-free word of length at least 10 has a prefix that is an Abelian square. Secondly, by Lemma~3 in Allouche et al.~\cite{AllCurSha1998}, if an infinite word \bfx is overlap-free, there exist a finite word $p\in \nsset{\epsilon, 0, 1, 00, 11}$  and an overlap-free infinite word $\bfy$ such that $\bfx = p \mu(\bfy)$. From this it follows that \bfx contains the words $\mu^n(0)$ for all $n \geq 1$. In other words, every subword of the Thue-Morse word is in \bfx, and consequently there is no number $\ell$ such that every position in \bfx has an occurrence of an Abelian cube with period at most~$\ell$.
\end{proof}

\section{Avoiding Abelian squares in the beginning}\label{201003121129}

In this section, we improve the second item of Corollary~\ref{103125012010}.
We refer the reader to  Chapter~10 of Lothaire~\cite{Lothaire2005} for the basic notions and properties  of morphisms left undefined here.

\begin{theorem}\label{134801022010}
There exists a uniformly recurrent infinite word with bounded Abelian complexity that does not begin in an Abelian square.
\end{theorem}
\begin{proof}
Consider the fixed point \bfz of the morphism $g \colon 0 \mapsto 0111110, 1\mapsto 01110$. 
The  incidence matrix of $g$ is $\left(\begin{smallmatrix} 2  & 2 \\ 5 & 3 \end{smallmatrix}\right)$ with eigenvalues
\[
\lambda_1 = \frac{1}{2}\bparen{ 5 +  \sqrt{41}} = 5.70\ldots \qtext{and} \lambda_1 = \frac{1}{2}\bparen{5 - \sqrt{41}} = - 0.70 \ldots.
\]
Therefore $g$ is a Pisot morphism, and by a result of Adamczewski~\cite{Adamczewski2003}, the word \bfz is $C$-balanced for some~$C$.
Thus according to Lemma~\ref{220220101341}, it also has  bounded Abelian complexity. 
Since $g$ is a primitive morphism, $\bfz$ is also uniformly recurrent (see also Remark~\ref{134401022010}).
Therefore we only need to show that \bfz does not have a prefix that is an Abelian square. 
Suppose that $uv$ is a prefix of \bfz with $u$ and $v$ Abelian equivalent. 
Since $\nabs{uv}_0$ is even, we see that $uv = h(w)$ for some word $w$. Furthermore, since $\nabs{uv}_{1}$ is even, it follows that $\nabs{w}$ is even,
so that there exist $w_{1}$ and $w_{2}$ such that $w = w_1 w_2$ and $\nabs{w_1} = \nabs{w_2}$. Since $\nabs{g(0)}_0 = 2$ and $\nabs{g(1)}_0 = 2$,
we have $\nabs{g(w_1)}_0 = \nabs{g(w_2)}_0$, and thus $u = g(w_1)$ and $v = g(w_2)$. Furthermore, since $\nabs{g(0)}_1 > \nabs{g(1)}_1$ it follows that the number of 0s and 1s in the words 
$w_1$ and $w_2$ is the same; in other words they are  Abelian equivalent. Therefore \bfz has a prefix that is shorter than $uv$ and an Abelian square. The claim now follows by induction.
\end{proof}

\begin{remark} 
The infinite word $\bfz$ constructed in the proof of the previous theorem avoids Abelian squares only in the beginning:
There is an Abelian square in every position except in first one.
\end{remark}

\section{Morphisms forcing bounded Abelian complexity}\label{201003121131}

It is an immediate fact (see also Theorem~\ref{220220101336}) that the Thue--Morse morphism maps all words to words with bounded Abelian complexity. 
We characterize the class of all morphisms sharing this property. This result will be useful in Section~\ref{201003121132} for providing an example of a word having infinitely many
positions without Abelian squares.

For a vector $\vec{v} = (v_1, \ldots, v_k) \in \N^k$, we denote $\nnorm{\vec{v}} = \sum_{i = 1}^k |v_i|$.

\begin{theorem}\label{201003101457}
A morphism $\map{f}{A}{B}$ maps all words to words with bounded Abelian complexity if and only if there exists $\vec{v} \in \N^{\#B}$ such that, for each letter $a \in A$, there exists an integer $K_a$ such that $\Psi(f(a)) = K_a \vec{v}$.
\end{theorem}

\begin{proof}
Suppose that $f(\bfw)$ has bounded Abelian complexity for all infinite words $\bfw$.
Consider two different letters, say $0$ and $1$, and denote $k = |f(0)|$, $\ell = |f(1)|$.
Assume first that $\Psi(f(0^\ell)) \neq \Psi(f(1^k))$ and observe that $|f(0^\ell)| = |f(1^k)|$ ($= |f(0)||f(1)|$).
Now observe that $\lim_{n \to \infty} \nnorm{\Psi(f(0^{n\ell})) - \Psi(f(1^{nk}))} = \infty$. So the word $\bfw = f(\prod_{i \geq 0} 0^i1^i)$ is not $C$-balanced for any integer~$C$, and  therefore by Lemma~\ref{220220101341}, the Abelian complexity of $\bfw$ is not bounded. Thus if the morphism $f$ maps any word to a word with bounded Abelian complexity, then for all letters $a, b \in A$,  we must have $\Psi(f(a^{ \nabs{f(b)} })) = \Psi(f(b^{\nabs{f(a)}} ))$, that is, $|f(b)| \Psi(f(a)) = |f(a)| \Psi(f(b))$. Now for each letter $a \in A$, let $K_a$ denote the $\gcd$ of the entries of $\Psi\bparen{f(a)}$; then we have $\Psi(f(a)) = K_a \vec{v}_a$ for some vector 
$\vec{v}_a \in \N^{\# B}$. Then $\nabs{f(b)} \Psi(f(a)) = \nabs{f(a)} \Psi(f(b))$ implies that $\vec{v}_a = \vec{v}_b$ for all letters $a$, $b$.

Conversely, suppose that there exists $\vec{v} \in \N^{\#B}$ such that for all letters $a \in A$, there exists an integer $K_a$ such that $\Psi(f(a)) = K_a \vec{v}$. If $\nnorm{\vec{v}} = 0$, then $f(\bfw)$ is empty for all words $\bfw$,  and the claim trivially holds. So let us assume that $\nnorm{ \vec{v} } \neq 0$.

Let $\bfw$ be an infinite word, and denote $M = \max \ncset{\nabs{f(a)}}{a \in A}$. If $w$ is a subword of \bfw longer than $M$, then it is of the form $sf(u)p$ with $s, u, p$ three words such that $s$ is a suffix of $f(a)$ for a letter $a$, $p$ is as prefix of $f(b)$ for a letter $b$. Observe that 
\[
\Psi(w) = \Psi(f(u)) + \Psi(sp) = \left[\sum_{a \in A} |u|_aK_a\right] \vec{v} + \Psi(sp).
\]
Denoting $C_u = \sum_{a \in A} \nabs{u}_aK_a$, we  thus have $\nabs{w} = C_u \nnorm{\vec{v}} + \nnorm{\Psi(sp)}$.
Then 
\[
\frac{\nabs{w}-2M}{ \nnorm{ \vec{v} }} \leq C_u \leq \frac{\nabs{w}}{ \nnorm{\vec{v} }},
\]
and consequently,
\[
\frac{\nabs{w} - 2M}{\nnorm{\vec{v}}} \vec{v} \leq  \Psi(w) \leq   \frac{\nabs{w}}{ \nnorm{\vec{v}}} \vec{v}  + 2 M \vec{1},
\]
where the inequality means coordinate-wise inequality and $\vec{1} = (1,1,\ldots, 1)$.
It follows that the word $f(\bfw)$ has bounded Abelian complexity.
\end{proof}

\section{Morphisms preserve bounded  Abelian complexity}\label{201003121130}

Having shown in the previous section that the image of an infinite word $\bfw$ under a morphism can have unbounded Abelian complexity, we show that this cannot happen when  the word $\bfw$ itself
has bounded Abelian complexity.

\begin{theorem}
\label{P:morphisms}
Let $A$ and $B$ be two alphabets, let $\bfw$ be an infinite word over $A^*$, and let $\map{f}{A^{*}}{B^{*}}$ be a morphism. If the Abelian complexity of $\bfw$ is bounded, then the Abelian complexity of $f(\bfw)$ is bounded.
\end{theorem}

\begin{remark}
No restrictions are imposed on the morphism~$f$. In particular,  it could be erasing. In this case the word $f(\bfw)$ may be finite,
but then its Abelian complexity is trivially bounded.
\end{remark}

\begin{proof}[Proof of Theorem~\ref{P:morphisms}]
The claim is trivially true if $f(\bfw)$ is finite, so let us assume that $f(\bfw)$ is infinite.
By assumption there exists an integer $K$ such that $\rho^{ab}_{\bfw}(n) \leq K$ for all $n\geq 1$. It follows from Lemma~\ref{220220101341} that $\bfw$ is $C$-balanced for some integer $C\geq1$.  Let us denote
\begin{align*}
K_1 &= C \sum_{a \in A} |f(a)|, \\
M &= \max \bcset{  \nabs{f(a)} }{a \in A}, \\
K_2 &= \max\bcset{ \nabs{y}  }{  \text{$y$ is  a subword of $\bfw$ and }  \nabs{f(y)} \leq K_1+M}, \\
K_3 &= K(M\#A)^2(K_2+1).
\end{align*}

Since $f(\bfw)$ is infinite, we have $K_1 \neq 0$ and $M \neq 0$. We need to prove that $K_2$ exists. Clearly,  the set defining $K_{2}$ is not empty. 
Let $a$ be a letter that occurs infinitely often in $\bfw$ and such that $f(a) \neq \varepsilon$.
Let $m$ be the minimal length of subwords of $\bfw$ containing at least $C+1$ occurrences of $a$. Since $\bfw$ is $C$-balanced, any subword of $\bfw$ of length $m$ contains at least one occurrence of $a$, and so any subword of length $m(K_1+M+1)$  contains at least $K_1+M+1$ occurrences of the letter $a$: for any such subword $y$, $|f(y)| > K_1 + M$. Thus the set defining $K_{2}$ is finite and therefore the number $K_{2}$ exists.

We will show that  $\rho^{\rm ab}_{f(\bfw)}(n) \leq K_3$ for all $n \geq M$; this implies that the Abelian complexity of $f(\bfw)$ is bounded by 
$\max\bsset{K_3, \rho^{\rm ab}_{f(\bfw)}(0), \ldots, \rho^{\rm ab}_{f(\bfw)}(M-1)}$.

For any integer $n \geq 0$, let us denote by ${\calT}_n$ the set of triplets $(s, u, p)$ of words $u\in A^{*}$ and  $s,p \in B^{*}$ such that there exist two letters $\alpha$ and $\beta$ with $\alpha u \beta$ a subword of $\bfw$, $s$ a proper suffix of $f(\alpha)$, $p$ a proper prefix of $f(\beta)$ and $\nabs{sf(u)p} = n$. Observe that any subword of $f(\bfw)$ of length  $n\geq M$ can be decomposed $sf(u)p$ with $(s, u, p) \in {\calT}_{n}$. Denote also 
${\calS}_n = \bcset{ m }{ (s, u, p) \in {\calT}_n, |u| = m}$. 

\medskip

\noindent
\textbf{Claim.} For all $n \geq 0$, we have $\#{\calS}_n \leq K_2 + 1$.

\begin{proof}[Proof of the claim]
Assume that $m_1, m_2 \in {\calS}_n$ with $m_1 \geq m_2$, and let $(s_1, u_1, p_1)$ and $(s_2, u_2, p_2)$ be two triplets in ${\calT}_n$ such that 
$\nabs{u_1} = m_1$ and $\nabs{u_2} = m_2$. We decompose $u_1$ into $u_1 = xy$ with $|x| = m_2$. By definition of ${\calT}_n$, words $u_1$ and $u_2$ are subwords of $\bfw$ (and so $x$ is also a subword of $\bfw$). Therefore since $\bfw$ is $C$-balanced, we have $|u_2|_a \leq |x|_a + C$ for all letters  $a \in A$. Thus
\[
\nabs{f(u_2)} = \sum_{a \in A} \nabs{u_2}_a \nabs{f(a)} \leq \sum_{a\in A} (\nabs{x}_{a} + C) \nabs{f(a)} = \nabs{f(x)}+ K_1.
\]
As a consequence of the definition of ${\calT}_n$, we have $|f(u_2)| = n - |s_2p_2|$ and $|f(x)| = n - |f(y)| - |s_1p_1|$. Thus
$|f(y)| + |s_1p_1| - |s_2p_2| \leq K_1$ which implies $|f(y)| \leq K_1 + M$ and so, by definition of $K_2$, $|y| \leq K_2$.
Hence $0 \leq m_1 - m_2 \leq K_2$ and elements of ${\calS}_n$ can take at most $K_2+1$ different values. This ends the proof of the claim.
\end{proof}

We continue the proof of Theorem~\ref{P:morphisms}.
As noted before, if $n\geq M$, then $\rho^{ab}_{f(\bfw)}(n)$ equals the number of Parikh vectors of words of the form $sf(u)p$ with $(s, u, p) \in 
{\calT}_n$. Now observe that if $(s, u_1, p)$ and $(s, u_2, p)$ are in ${\calT}_n$ with $\Psi(u_1) = \Psi(u_2)$, then $\Psi(sf(u_1)p) = \Psi(sf(u_2)p)$.
Therefore the quantity $\rho^{ab}_{f(\bfw)}(n)$ is bounded by the number of triplets of the form  $(s, \Psi(u), p)$  with $(s, u, p) \in {\calT}_n$.
By the previous claim, we know that a word $u$ such that a triplet of the form $(., u,.)$ is in $\calT_{n}$ can take at most $K_{2} + 1$ different lengths.
Thus by hypothesis, we get at most $K(K_2+1)$ different possible vectors $\Psi(u)$. Moreover there are at most $M\#A$ possibilities for $s$ and  at most 
$M\#A$ possibilities for $p$. Hence the cardinality of $\bigl\{(s, \Psi(u), p) : (s,u,p) \in \calT_n\bigr\}$ is bounded by $K_3$, and the Abelian complexity of $f(\bfw)$ is therefore bounded.
\end{proof}

\begin{remark} 
The converse of Theorem~\ref{P:morphisms} does not hold; in fact, by Theorem~\ref{220220101336}, the Abelian complexity of $\mu(\bfw)$ 
is bounded for any binary word~$\bfw$.
\end{remark}

\section{Avoiding Abelian squares in infinitely many positions} \label{201003121132}

\begin{theorem}\label{P:Ab2}
There exists a uniformly recurrent infinite word with bounded Abelian complexity 
in which there are infinitely many positions where no Abelian square occur.
\end{theorem}

To prove this theorem, we first state an important property of the uniform morphism $\map{f}{\binm}{\binm}$ defined by
\[
f(0) = 00011  \qqtext{and} f(1) = 01100.
\]

\begin{lemma}\label{L:f_and_ab_squares}
Let $\bfw$ be an infinite binary word, and suppose that $f(\bfw)$ begins in a word of the form $0001uv$, where $u$ and $v$ are nonempty Abelian equivalent words.  Then $\bfw$ has a prefix of the form $0x0y0$ for some words $x$ and $y$ with $|x|=|y|$. 
\end{lemma}

\begin{proof} We divide the proof into five cases depending on the remainder of $\nabs{u}$ mod~5:
\begin{enumerate}
\item We have $\nabs{u} \equiv 0 \pmod{5}$. Then, since $\nabs{u} \neq 0$, the word $u$ is of the form $1 f(x) 0001$ or $1f(x) 0110$. In the first case $v$ must be of the form $1f(y) 0001$ or $1f(y) 0110$ with $\nabs{x} = \nabs{y}$; in the second case $v$ is of the form $0f(y) 0001$ of $0f(y) 0110$. 
In both cases, $|x| = |y|$ so $f(x)$ and $f(y)$ are Abelian equivalent. 
Since $u$ and $v$ are Abelian equivalent, it follows that the only possible case is $u = 1 f(x) 0001$ and $v = 1 f(y) 0001$. Then $\bfw$ has prefix $0x0y0$,  and the claim holds.
\item We have $\nabs{u} \equiv 1 \pmod{5}$. Then $u$ is of the form $u = 1 f(x)$ and $v$ is of the form $v = f(y) 0$ for some words $x$ and $y$. This  is, however, a contradiction because $u$ and $v$ are Abelian equivalent.
\item We have $\nabs{u} \equiv 2 \pmod{5}$. Then $u$ is of the form $u=1f(x)0$ and $v$ has one of the following four forms:
\[
0011 f(y) 000,  \qquad   0011f(y) 011, \qquad 1100 f(y) 000, \qquad 1100 f(y) 011.
\]
But regardless of which of these forms $v$ may have, the words $u$ and $v$ are not Abelian equivalent, a contradiction.
\item We have $\nabs{u} \equiv 3 \pmod{5}$. Then either  $u$ is of the form $1 f(x) 00$ and $v$ is of the form $011f(y)$, or $u$ is of the form $1 f(x) 01$ and $v$ is of the form
$100 f(y)$. But again we quickly verify that $u$ and $v$ cannot be Abelian equivalent in either situation, a contradiction.
\item We have $\nabs{u} \equiv 4 \pmod{5}$. We separate two possibilities: Either $u=1 f(x) 000$ or $u = 1 f(x) 011$. In the first situation, we have either $v = 11 f(y) 00$ or $11 f(y) 01$;
in the latter situation we have either $v = 00 f(y) 00$ or $v =  00 f(y) 01$. But, as in the previous cases, we  see that $u$ and $v$ cannot be Abelian equivalent. This  contradiction concludes the proof.
\end{enumerate}
\end{proof}

\begin{proof}[Proof of Theorem~\ref{P:Ab2}]
Let $h$ be the uniform morphism defined by $h(0) = 01011111$ and $h(1) = 11101111$, and let $\bfw_h$ be the fixed point of $h$ beginning with the letter~$0$.
Our goal is to show that the infinite word $f(\bfw_h)$ is a uniformly recurrent word avoiding Abelian squares in infinitely many positions.
To achieve this, we need the following four steps:
\begin{enumerate}
\item Let $\bfw$ be any binary word beginning with the letter~$0$, and suppose that the word $h(\bfw)$ has a prefix of the form $0x0y0$ with $|x| = |y|$. 
Then there exist words $x'$ and $y'$ with $|x'| = |y'|$ such that $0x = h(0x')$, $0y = h(0y')$, and $0x'0y'0$ is a prefix of $\bfw$.
\begin{proof}
Let the letters in $h(\bfw)$ be indexed starting from~1. Since $\nabs{h(0)} = \nabs{h(1)} = 8$, the way the letter 0 occurs in $h(0)$ and $h(1)$ implies that 0 can occur in $h(\bfw)$ only at positions congruent to 1, 3, or  4 $\pmod 8$. Therefore $\nabs{0x}$ is congruent to 0, 2, or~3 $\pmod 8$;  accordingly $\nabs{0x0y0}$  must be congruent to 1, 5, or 7 $\pmod 8$, respectively.  Only the first case is possible, and thus we have $0x = h(0x')$, $0y = h(0y')$ for some $x', y'$ with $\nabs{x'} = \nabs{y'}$ such that $0x'0y'0$ is a prefix of $\bfw$.
\end{proof}
\item The word $\bfw_h$ does not have a prefix of the form $0x0y0$ with $|x| = |y|$.
\begin{proof}
Immediate consequence of the first step.
\end{proof}
\item The word $\bfw_h$ does not have a prefix of the form $010x0y0$ with $|x| = |y|$.
\begin{proof}
Suppose that $010x0y0$ with $\nabs{x} = \nabs{y}$ is a prefix of $\bfw_h$.
Since the letter 0 occurs only in positions congruent to 1, 3, or 4 $\pmod 8$, we readily check that $\nabs{0x} = \nabs{0y} \equiv 0 \pmod8$.
Hence the fact that $010x0$ is a prefix of $\bfw_h$ implies that $010x0 = h(0x')010$ for some $x'$.
Further, since $010x0y0$ is a prefix of $\bfw_h$ and $010x 0 y 0 =  h(0x') 010 y 0$, it follows that $010y0 = h(0 y') 010$ for some $y'$.
Consequently we have,
\[
010x0y0 = h(0x') 010 y 0 = h(0 x') h(0 y') 010,
\]
and so $\bfw_h$ has prefix $h(0x' 0y'0)$ and thus a prefix $0x'0y'0$, where $\nabs{x'} = \nabs{y'}$. This  contradicts the previous step.
\end{proof}
\item The word $\bfw_h$ does not have a prefix of the form $h^n(01)0x0y0$ with $|x| = |y|$.
\begin{proof}
Suppose that $\bfw_h$ does have prefix $h^n(01)0x0y0$ with $|x| = |y|$. By the previous step, we have $n\geq1$.
Denote $\bfw_h = 01 \bfz$. Since $h^{n}(\bfw_{h}) = \bfw_{h}$, we see that $h^n(\bfz)$ has prefix $0x0y0$, and so by the first step, the word $h^{n-1}(\bfz)$ 
has prefix $0x' 0y'0$ for some $x', y'$ with $\nabs{x'} = \nabs{y'}$.
Consequently $\bfw_h$ has a prefix of the form $h^{n-1}(01) 0x' 0 y' 0$, which is a contradiction by induction.
\end{proof}
\end{enumerate}

Finally we are ready to wrap up the proof of Theorem~\ref{P:Ab2}.
The word $\bfw_h$ is uniformly recurrent since it is a fixed point of a primitive morphism. Thus $f(\bfw_h)$ is uniformly recurrent as well. 
Its Abelian complexity is bounded by Theorem~\ref{201003101457}.
For $n \geq 0$, let $\bfw_n$ be the word determined by $\bfw_h = h^n(01)\bfw_n$.
By the previous step, $\bfw_n$ does not have a prefix of the form $0x0y0$ with $|x| = |y|$.
According to Lemma~\ref{L:f_and_ab_squares}, the word $f(\bfw_n)$ does not have a prefix of the form $0001uv$ with $u$ and $v$ Abelian equivalent. 
This shows that there exist infinitely many positions in $f(\bfw_h)$ in which no Abelian square occurs.
\end{proof}

\begin{remark}\label{134401022010}
An infinite word is \emph{linearly recurrent} if, for all subwords $w$, any two consecutive occurrences of $w$ are within $L\nabs{w}$ positions, 
where $L>0$ is constant. The infinite words constructed in the proofs of Theorems~\ref{134801022010} and~\ref{P:Ab2} are not only uniformly, 
but even linearly recurrent.  In fact, F.~Durand~\cite{Durand1998} showed that every fixed point of a primitive morphism is linearly recurrent
(see also~\cite{Durand2008} for a precise statement of this).  
\end{remark}

\section{The general case} \label{201003121133}

The set of positions in the infinite word  $f(\bfw_{h})$ given in the proof of Theorem~\ref{P:Ab2} 
in which no Abelian squares occurs has density~$0.$
This suggests the following question: Does there exist an infinite word with bounded Abelian complexity avoiding Abelian cubes in a set of positions with positive density?  The answer is no. 
Indeed, by using Szmer\'edi's theorem~\cite{Szemeredi1975}  instead of van der Waerden's theorem in the proof of Theorem~5.1 in~\cite{RicSaaZam2010}, we get the following result:
\begin{theorem}
Let \bfw be an infinite word with bounded Abelian complexity, and let $D\subset \N$
be a set of natural numbers with positive upper density, that is 
\[
\limsup_{\ntoinf} \frac{D \cap \nsset{1,2, \ldots, n}  }{n} >0.
\] 
Then, for every integer $k\geq 2$,  there exists an integer $i\in D$ such that there is an Abelian $k$-power occurring at position $i$ in \bfw.  
\end{theorem}

\noindent This result naturally leads to the following question.

\begin{openproblem}\label{223720042010}
Does there exist  an infinite binary word that avoids Abelian squares in positions with positive upper density? 
\end{openproblem}

Let us consider a slightly modified version of this question. We say that an infinite binary word $\bfw$ satisfies \emph{BAS property} if
the set of positions in \bfw that avoid Abelian squares have positive lower density. This means that such positions in \bfw occur in bounded gaps.
In what follows we show that the existence of a word with BAS property is equivalent to the following well-known
question asked independently  by Pirillo--Varricchio~\cite{PirVar1994} and Halbeisen--Hungerb\"{u}hler~\cite{HalHun2000}:
\begin{quote}  
\textit{Does there exist an infinite word  over  a finite set of  integers such that no two consecutive blocks of the same length have the same sum?}
\end{quote}

We say that an infinite word satisfying the above condition has the \emph{PVHH property}. 

\begin{theorem}
There exists an infinite word satisfying the PVHH property if and only if there exist an infinite binary word satisfying the BAS property. 
\end{theorem}

We prove the claim in the next five lemmas. If $x$ is a finite word whose letters are integers, let $\sum x$ denote the sum of the letters in~$x$.

\begin{lemma}
If there exists a word satisfying the PVHH property, 
then there exists an infinite binary word satisfying the BAS property.
\end{lemma}
\begin{proof}
Suppose that an infinite word $\bfx = x_{1} x_{2} \ldots $ satisfies the PVHH property.
Suppose that $\bfx$ is over an alphabet $\calA = \nsset{a_{1}, a_{2}, \ldots, a_{k}}$, where each $a_{i}$ is an integer.
Since the set of words satisfying the PVHH property is clearly closed under affine transformations, 
we may assume that each $a_{i}$ is an odd positive integer.
Define a morphism $\tau$ by $\tau(a_{i}) = 01^{a_{i}}0$. We claim that the word $\tau(\bfx)$ has the BAS property.
We prove this by showing that, for all $i=1,2,\ldots$, the word $\bfz := \tau(x_{i} x_{i+1} \cdots)$ does not have an Abelian square as a prefix.

Suppose the contrary; a word $uv$ is a prefix of \bfz with $u$ and $v$ Abelian equivalent. 
We proceed by the following reasoning analogous to what we used in the proof of Theorem~\ref{134801022010}.
Since $\nabs{uv}_0$ is even, the form of the $\tau(a_{i})$ implies that $uv = \tau(w)$ for some word $w\in \calA$. Similarly, since $\nabs{uv}_{1}$ is even, it follows that $\nabs{w}$ is even,
and we can write $w = w_1 w_2$ with $\nabs{w_1} = \nabs{w_2}$. Since $\nabs{\tau(a_{i})}_0 = 2$ for all $i$, 
we have $\nabs{\tau(w_1)}_0 = \nabs{\tau(w_2)}_0$, and thus $u = \tau(w_1)$ and $v = \tau(w_2)$. But now,
\[
\nabs{u}_{1} = \sum_{i=1}^{k} a_{i} \nabs{w_{1}}_{a_{i}} = \sum w_{1}  \qqtext{and} \nabs{v}_{1}  = \sum_{i=1}^{k} a_{i} \nabs{w_{2}}_{a_{i}} = \sum w_{2}
\]
which contradicts the PVHH property because $\nabs{u}_{1} = \nabs{v}_{1}$.
\end{proof}

We begin proving the converse by defining the following word. Let $M$ be a positive integer, and let $\bfv = v_{0} v_{1} \ldots$ be the periodic sequence given by
\[
v_{n} = 
\begin{cases}
2^{\, n \bmod M} & \text{if $n \not\equiv -1 \pmod{M}$} \\
1 - 2^{M-1} & \text{otherwise}.
\end{cases}
\]

\begin{lemma}\label{221920042010}
Let $x,y$ be two factors of \bfv. If $\sum x = \sum y$, then $\nabs{x} \equiv \nabs{y} \bmod{M}$.
\end{lemma}
\begin{proof}
Observe first that $\sum_{n=0}^{i-1} v_{n} = 2^{i \bmod M}-1$. Thus if $x$ occurs at position $i$ and has length $j-i$,
then $\sum x = 2^{j \bmod M} - 2^{i \bmod M}$.
If $\sum x = 0$, then $|x| \equiv 0 \pmod M$.
Otherwise, the congruence classes mod $M$ of $i$ and $j$, hence of $|x|$,
can be recovered from $\sum x$.
\end{proof}

\begin{lemma}
Let $\bfu = u_{0} u_{1} \ldots $ be any infinite binary word, and define a word $\bfw=w_{0} w_{1} \ldots$ by
\[
w_{n} = u_{n} \cdot 2^{M} + v_{n}.
\]
If $xy$ is a factor of \bfw, say $x = w_{i} w_{i+1} \cdots w_{j-1}$ and $y = w_{j} w_{j+1} \cdots w_{k-1}$, 
with $\sum x = \sum y$, then $\nabs{x} \equiv \nabs{y}\pmod{M}$ and
\begin{equation}\label{115317042010}
\sum_{h=i}^{j-1} u_{h} = \sum_{h=j}^{k-1} u_{h}. 
\end{equation}
\end{lemma}

\begin{proof}
First off, we have
\[
\sum x = 2^{M} \cdot \sum_{h=i}^{j-1} u_{h} + \sum_{h=i}^{j-1} v_{h} \qqtext{and}
\sum y = 2^{M} \cdot \sum_{h=j}^{k-1} u_{h} + \sum_{h=j}^{k-1} v_{h}.
\]
Since $\sum x = \sum y$, this gives
\[
2^{M} \cdot \Bparen{\sum_{h=i}^{j-1} u_{h} - \sum_{h=j}^{k-1} u_{h}} + \Bparen{\sum_{h=i}^{j-1} v_{h}  - \sum_{h=j}^{k-1} v_{h}} = 0.
\]
The inequality
\[
\Babs{\sum_{h=i}^{j-1} v_{h}  - \sum_{h=j}^{k-1} v_{h}} \leq \Babs{\sum_{h=i}^{j-1} v_{h}} + \Babs{\sum_{h=j}^{k-1} v_{h}} \leq 
2 \cdot\bparen{2^{M-1} -1} < 2^{M}
\]
thus implies that
\[
\Bparen{\sum_{h=i}^{j-1} u_{h} - \sum_{h=j}^{k-1} u_{h}} = 0 \qqtext{and} \Bparen{\sum_{h=i}^{j-1} v_{h}  - \sum_{h=j}^{k-1} v_{h}} = 0.
\]
The first equation gives~\eqref{115317042010}, and the latter equation with Lemma~\ref{221920042010} implies that $\nabs{x} \equiv \nabs{y} \pmod{M}$.
\end{proof}

\begin{lemma}\label{003218042010}
Let $N$ be a positive integer. Suppose that \bfu is a binary word with positions $n_{0}, n_{1}, n_{2}, \ldots$ such that 
$iN \leq n_{i} < (i+1) N$ and each position $n_{i}$ avoids Abelian squares.
Define a new word \bfz by
\[
z_{i} = w_{n_{i}} +  w_{n_{i} +1} + \cdots + w_{n_{i+1}-1},
\]
where \bfw is the word defined in the previous lemma with $M = 2N$.
Then \bfz satisfies the PVHH property.
\end{lemma}

\begin{proof}
Suppose that $xy$ is a factor of \bfw such that $\sum x = \sum y$ and $\nabs{x} = \nabs{y}$. 
We have
\[
z_i z_{i+1} \cdots z_{i+\ell-1} \qqtext{and} y = z_{i+\ell} z_{i+\ell+1} \cdots z_{i+2\ell-1} 
\]
for some $\ell\geq 1$. Let $x'$ and $y'$ denote the corresponding decoded words in \bfw, that is to say, 
let
\[
x' = w_{n_{i}} w_{n_{i}+1}\cdots w_{n_{i+\ell}-1} \qqtext{and} y' = w_{n_{i+\ell}} w_{n_{i+\ell}+1}\cdots w_{n_{i + 2\ell}-1}.
\]
Since $\sum x' = \sum y'$, the previous lemma says that $\nabs{x'} \equiv \nabs{y'} \pmod{2N}$
and
\begin{equation}\label{004018042010}
\sum_{h=n_{i}}^{n_{i+\ell}-1} u_{h} = \sum_{h=n_{i + \ell}}^{n_{i + 2 \ell}-1} u_{h}. 
\end{equation}
On the other hand, $\nabs{x'} = n_{i+\ell} - n_{i}$ and $\nabs{y'} = n_{i + 2\ell} - n_{i+\ell}$, so that
\[
-2 N < \nabs{y'} - \nabs{x'} < 2 N.
\]
Therefore $\nabs{x'} = \nabs{y'}$. But this and Eq.~\eqref{004018042010} means that there is an 
Abelian square of length $\nabs{x'y'}$ in \bfu occurring at position $n_{i}$, a contradiction.
\end{proof}

\begin{lemma}
If there exists an infinite binary word satisfying the BAS property, then there exists an infinite word satisfying the PVHH property.
\end{lemma}
\begin{proof}
If \bfu is an infinite binary word with the BAS property, then it satisfies conditions of the previous lemma for some integer $N$, and therefore a word satisfying the PVHH property exists.
\end{proof}

\section*{Acknowledgements}

The third author is supported by grant no.~134190 from the Finnish Academy.
The fourth author is partially supported by grant no.~090038011 from the Icelandic Research Fund, and by the grant SubTile funded by the A.N.R.

\end{document}